\documentclass[a4paper,twoside,11pt,reqno]{amsart}
\usepackage[margin=1in]{geometry}
\usepackage[english]{babel}
\usepackage[utf8]{inputenc}
\usepackage{amsmath}
\usepackage{amssymb}
\usepackage{amsfonts}
\usepackage{amsthm}
\usepackage{bm}
\usepackage{mathrsfs}
\usepackage{hyperref}
\usepackage{comment}
\hypersetup{
	colorlinks=true,
	linkcolor=blue,
	filecolor=blue,
	citecolor=purple,      
	urlcolor=cyan,
}
\usepackage{xcolor}
\usepackage{dsfont}
\usepackage{physics}
\usepackage{stmaryrd}
\usepackage{enumitem}

\usepackage{tikz-cd}
\usepackage{tikz}
\usepackage{tikzsymbols}
\usepackage{mathtools}
\usetikzlibrary{matrix}

\usepackage[matrix,arrow,cmtip]{xy} 
\SelectTips{cm}{}
\newdir{(}{{}*!/-5pt/@^{(}}
\newdir{(x}{{}*!/-5pt/@_{(}}
\newdir{+}{{}*!/-9pt/{}}
\newdir{>+}{@{>}*!/-9pt/{}}
\entrymodifiers={+!!<0pt,\fontdimen22\textfont2>}

\theoremstyle{plain}
\newtheorem{theorem}[subsubsection]{Theorem}

\newtheorem{lemma}[subsubsection]{Lemma}

\newtheorem{corollary}[subsubsection]{Corollary}
\theoremstyle{definition}
\newtheorem{definition}[subsubsection]{Definition}
\newtheorem{definition-proposition}[subsubsection]{Definition/Proposition}

\newtheorem{question}[subsubsection]{Question}
\newtheorem{example}[subsubsection]{Example}
\theoremstyle{remark}
\newtheorem{remark}[subsubsection]{Remark}

\numberwithin{equation}{section}

\newcommand{\ZZ}{\mathbb{Z}}
\newcommand{\OO}{\mathcal{O}}

\newcommand{\Hom}{\mathrm{Hom}}

\newcommand{\Coh}{\mathrm{Coh}}

\newcommand{\cA}{\mathcal{A}}

\newcommand{\cC}{\mathcal{C}}

\newcommand{\cB}{\mathcal{B}}

\newcommand{\cH}{\mathcal{H}}

\newcommand{\too}{\longrightarrow}

\newcommand{\Db}{\mathrm{D}^{b}}
\newcommand{\KGdim}{\mathrm{KGdim}}
\newcommand{\cD}{\mathcal{D}}
\newcommand{\cT}{\mathcal{T}}
\newcommand{\Dp}{\mathrm{D}^{\mathrm{perf}} }

	\title{Geometric phantom categories do not admit Noetherian t-structures}
	\date{}
	\author{Yeqin Liu}

\begin{document}
	
	\maketitle
	
	\begin{abstract}
There are no Noetherian or Artinian bounded t-structures on geometric phantom or quasi-phantom categories, in the sense of \cite{GO13}.
	\end{abstract}
	
	\section{Introduction}

\begin{definition}[Slight generalization from \cite{GO13}]\label{def-phantom}
	Let $X$ be a quasi-compact, quasi-separated scheme. An admissible subcategory $\cC\subset \Dp(X)$ is a \emph{geometric phantom category} if $K_{0}(\cC)=0$. It is a \emph{geometric quasi-phantom category} if $K_{0}(\cC)$ is a torsion group.
\end{definition}

The presence of (quasi-)phantom categories is considered as a pathology. For instance, it can obstruct the Jordan-H\"{o}lder property of semi-orthogonal decompositions of triangulated categories \cite{BGS14}, and raise problems in finding Bridgeland stability conditions. It was once believed in the literature that geometric (quasi-)phantom categories do not exist. However, phantom categories seem ubiquitous (e.g. \cite{GO13, AO13, BGKS15, Kra24}), and a tool is developed to detect such categories \cite{Kuz15}.

 To formulate a reasonable moduli problem on a triangulated category, the existence of Bridgeland stability \cite{Bri07} or polynomial stability \cite{Bay09} conditions is very useful. However since the Grothendieck groups of (quasi-)phantom categories are torsion, there cannot be a central charge on such categories, and therefore Bridgeland/polynomial stability conditions cannot exist. By \cite{HL23}, phantom categories can also obstruct stability conditions on smooth proper noncommutative schemes with rank one $K_{0}$.

In the study of triangulated categories, bounded t-structures are useful tools and define the fundamental structures of the categories. 
The notion of stability conditions can be defined by a slicing \cite{Bri07, Bay09}, which always yields a bounded t-structure. Hence for (quasi-)phantom categories, although there is no stability conditions on them, one may ask the following natural question, which is addressed in \cite{Sos20}:

\begin{question}
    Can there exist a bounded t-structure on a (quasi-)phantom category?
\end{question}

 The following theorem partially answers the question: even if there is a bounded t-structure, it does not behave well.  The proof also works for any small triangulated category $\cC$ with $K_{0}(\cC)=0$ (or torsion) with a classical generator.
 
\begin{theorem}[Theorem \ref{theorem-main}]\label{theorem-intro-main}
    Geometric phantom or quasi-phantom categories do not admit Noetherian or Artinian bounded t-structures. 
\end{theorem}

\subsection*{Acknowledgment}
My special thank goes to Shengxuan Liu, for contributing the proof of Lemma \ref{thick closure}, many motivating conversations, and polishments of this paper.  
I thank Alexander Perry and Haoyang Liu for many useful discussions, and Izzet Coskun for his suggestions on an early draft. 
	\section{Main result}

    \subsection{T-Structures and generators of triangulated categories}
    
First we recall the notion of a t-structure. 

	\begin{definition}[\cite{BBD82}]\label{def-t-str}
		Let $\cT$ be a triangulated category. A t-structure on $\cT$ is a pair of full subcategories $(\cT^{\leq0},\cT^{\geq0})$ such that 
            \begin{enumerate}
                \item For any objects $E\in\cT^{\leq0}$ and $F\in\cT^{\geq0}$, we have $\Hom(E[1],F)=0$.
                \item\label{2} The category $\cT^{\leq0}$ is closed under $[1]$. This implies that $\cT^{\geq0}$ is closed under $[-1]$.
                \item\label{3} For any object $E\in\cT$, there is an exact triangle
                $
                    \tau_{\leq-1}E\rightarrow E\rightarrow \tau_{\geq0}E\rightarrow \tau_{\leq-1}E,
                $
                such that $\tau_{\leq-1}E\in\cT^{\leq0}[1]$ and $\tau_{\geq0}E\in\cT^{\geq0}$. (The triangle can be shown unique.)
            \end{enumerate}
            Denote $\cT^{\leq n}:=\cT^{\leq0}[-n]$ and $\cT^{\geq n}:=\cT^{\geq0}[-n]$. The full subcategory $\cA:=\cT^{\leq 0} \cap \cT^{\geq 0}$ is called the \emph{heart} of the t-structure. 

            The t-structure is called \emph{bounded}, if the natural inclusion
            $\bigcup_{n\in \ZZ} \cT^{\leq n}\cap \cT^{\geq -n} \too \cT$
            is an equivalence.
            For the rest of this paper, by a \emph{t-structure} we refer to a bounded t-structure.

	\end{definition}

	\begin{example}
	We give some intuitions of t-structures by the following examples.
	\begin{itemize}
		\item Let $\cA$ be an abelian category. 
		Then $\cA$ is the heart of a t-structure of $\Db(\cA)$.
		\item In general, let $\cA\subset \cC$ be the heart of a bounded t-structure on $\cC$. Then $\Db(\cA)$ need not to be equivalent to $\cC$. See Example 27 of \cite{Sch11}.
	\end{itemize}
\end{example}

It is worth mentioning that \cite{AGH19} finds a K-theoretic obstruction of bounded t-structures on triangulated categories. However, this obstruction does not apply to our case: If $\cA\subset \Db(X)$ is an admissible subcategory of a smooth projective variety, then $K_{i}(\cA)$ is a direct summand of $K_{i}(\Db(X))$ for $i\in \ZZ$. Since $K_{i}(X)=0$ for every $i\leq -1$, we have $K_{i}(\cA)=0$ for every $i\leq -1$. 

Next we recall the notion of generators of triangulated categories.
	
	\begin{definition}\label{def-generator-triangulated}
		Let $\cD$ be a small triangulated category. A triangulated subcategory $\cC\subset \cD$ is \emph{thick}, if it is closed under direct summands. For a set of objects $S$, its \emph{thick closure} is the smallest thick full triangulated subcategory containing $S$. An object $G\in \cD$ is a \emph{classical generator}, if its thick closure is $\cD$. 
	\end{definition}

	\begin{theorem}[\cite{BV03}]\label{theorem-classicalgenerator}
		Let $X$ be a quasi-compact, quasi-separated scheme. Then $\Dp(X)$ admits a classical generator. 
	\end{theorem}
	
	\subsection{Preliminaries on abelian categories}
	
	\begin{definition}
		Let $\cA$ be an abelian category. An object $X$ is \emph{Noetherian} (resp. \emph{Artinian}) if every ascending (resp. descending) chain of subobjects of $X$ stabilizes. The category $\cA$ is \emph{Noetherian} (resp. \emph{Artinian}) if every object of $\cA$ is Noetherian (resp. Artinian). 
	\end{definition}

    Next we recall the notion of Serre subcategories of abelian categories, which allows us to define quotients of abelian categories.
	\begin{definition}\label{def-generator-abelian}
	Let $\cA$ be an abelian category. An abelian subcategory $\cA'\subset \cA$ is a \emph{Serre subcategory}, if it is closed under subobjects, quotient objects, and extensions.  We call an object $F\in \cA$ a \emph{generator}, if $F$ is not contained in any proper Serre subcategory. 
	\end{definition}

	\begin{definition-proposition}
		Let $\cB\subset \cA$ be a Serre subcategory. Then there is an abelian category $\cA/\cB$ called the \emph{quotient category}, and an exact functor $Q:\cA \to \cA/\cB$ surjective on objects, such that for $X\in \cA$, $Q(X)\cong 0$ if and only if $X\in \cB$. The functor $Q$ satisfies the following universal property:
        If $F:\cA \to \cC$ is an exact functor of abelian categories such that for $X\in \cB$ we have $F(X)\cong 0$, then there is a unique exact functor $G: \cA/\cB \to \cC$, such that $F=G\circ Q$.
	\end{definition-proposition}

Recall the following notion of simple objects in abelian categories.
    
	\begin{definition}\label{def-simpleobject}
		Let $\cA$ be an abelian category. An object $F$ is \emph{simple}, if its only subobjects are $F$ and 0. 
	\end{definition}
	
	\begin{example}\label{example-simple}
		The following are some examples of simple objects.
		\begin{enumerate}
			\item Let $\mathrm{Vect}_{k}$ be the category of vector spaces over a field $k$. Then $k$ is simple.
			\item Let $\Coh(X)$ be the category of coherent sheaves on an algebraic variety. Then the skyscraper sheaf $\OO_{x}, x\in X$ is simple. 
            \item Let $G$ be a finite group and $\mathrm{Rep}(G)$ be the category of finite dimensional representations of $G$. Then irreducible representations are simple.
		\end{enumerate}
	\end{example}

    The following definition is the key idea to prove the main theorem.
	
	\begin{definition}[\cite{Gab62, GJ81}] \label{def-KGdim}
		Let $\cA$ be an abelian category. For every ordinal $\alpha$, define the following Serre subcategories inductively:
		\begin{itemize}
			\item $\cA_{\leq -1}:=0$, 
		    \item $\cA_{\leq \alpha}$:= the Serre subcategory generated by objects that are simple (Definition \ref{def-simpleobject}) in $\cA/\cA_{\leq \alpha-1}$, if $\alpha$ has a predecessor,
		    \item $\cA_{\leq \alpha}:= \bigcup_{\alpha'<\alpha} \cA_{\leq \alpha'}$, if $\alpha$ is a limit ordinal. 
		\end{itemize}
		If $\cA=\cA_{\leq \alpha}$ for some ordinal $\alpha$, the \emph{Krull-Gabriel dimension} (denoted by $\KGdim(\cA)$) of $\cA$ is defined to be the minimal of such $\alpha$. Otherwise $\KGdim(\cA)$ is not defined.
	\end{definition}

The geometric intuition of Definition \ref{def-KGdim} is explained in the following example.
    
	\begin{example}
		Let $X$ be an algebraic variety. Using notations in Definition \ref{def-KGdim}, for every $d\in \ZZ_{\geq 0}$, we claim $$\Coh(X)_{\leq d}= \mbox{full subcategory whose objects are coherent sheaves } F \mbox{ with} \dim (supp(F))\leq d. $$ 
		We see the claim for $d=0$ in Example \ref{example-simple}. In general, let $F\in \Coh(X)$ with $\dim(supp(F))=d$. If the Fitting ideal of $F$ corresponds to an irreducible variety in $X$, then $F$ is simple in $\Coh(X)/\Coh(X)_{\leq d-1}$. Every coherent sheaf $E$ with $\dim(supp(E))=d$ admits a filtration in $\Coh(X)/\Coh(X)_{\leq d-1}$ whose graded pieces are simple (i.e. the sheaves $F$ as above). Hence the claim follows, and we see that
		$\KGdim(\Coh(X))=\dim(X). $
	\end{example}

 Noetherianity/Artinianity condition in Theorem \ref{theorem-intro-main} is only needed for the following lemma.
    
	\begin{lemma}\label{welldefineddim}
		Let $\cA$ be a Notherian or Artinian abelian category. Then $\KGdim(\cA)$ is defined. 
	\end{lemma}
	
	\begin{proof}
		See Lemma 14.1.13 of \cite{Kra22}.
	\end{proof}
    We know more about hearts of admissible subcategories of $\Dp(X)$.
	\begin{lemma}\label{predecessor}
		Let $\cA$ be an abelian category with $\KGdim(\cA)=\alpha$ defined. If $\cA$ admits a  generator, then $\alpha=(\alpha-1)+1$ has a predecessor $(\alpha-1)$ (i.e. not a limit ordinal).
	\end{lemma}
	
	\begin{proof}
		Let $G\in \cA$ be a generator. If $\cA=\bigcup_{\beta<\KGdim(\cA)} \cA_{\leq \beta}$, then $G\in \cA_{\leq \beta}$ for some $\beta<\KGdim(\cA)$. However, since $G$ is a generator, we have $\cA_{\leq \beta}=\cA$. Hence by definition, we have $\KGdim(\cA)\leq \beta$, a contradiction.
	\end{proof}

    \subsection{Proof of the main theorem}
    
It is crucial that a geometric phantom admits a generator.

        \begin{lemma}\label{thick closure}
            Let $\cT$ be a triangulated category, and let $\cA$ be the heart of a bounded t-structure of $\cT$. Let $\cA'\subset\cA$ be a Serre subcategory. Then the following three categories are equivalent:
            \begin{enumerate}
                \item\label{a} The smallest triangulated subcategory $\cB$ that contains $\cA'$.
                \item\label{b} The thick closure (Definition \ref{def-generator-triangulated}) $\cC$ of $\cA'$.
                \item\label{c} The category $\cD:=\{E|\cH^i(E)\in\cA'\}$.
            \end{enumerate}
        \end{lemma}

        \begin{proof}
            First note that $\cA'$ is closed under direct summand because it is a Serre subcategory of $\cA$.
    The equivalent of \ref{a} and \ref{b} is well known (e.g. \cite{BV03}). We show that these are equivalent to \ref{c}. For any object $E\in\cD$, we have $E\in\cB$. We prove the converse by induction. Suppose $E\in\langle\cA'\rangle_1$, then $E$ is a direct summand of a finite direct sum of shifts of objects inside $\cA'$. Then $\cH^i(E)$ is a direct summand of finite direct sum of objects inside $\cA'$. Because $\cA'$ is closed under direct summand, we have $E\in\cD$. Now suppose $E_1\in\langle\cA'\rangle_1$, $E_2\in\langle\cA'\rangle_k$, which we know that $E_1,E_2\in\cD$. Then for any extension $E$ of $E_1,E_2$, we have an exact sequence
            $$
            \rightarrow\cH^i(E_1)\rightarrow \cH^i(E)\rightarrow \cH^i(E_2)\rightarrow
            $$
            for every $i\in \ZZ$.
            Because $\cA'$ is a Serre subcategory, we have $\cH^i(E)\in\cA'$. Because $\cA'$ is closed under direct summand, we know that any direct summand $F$ of $E$ also has the property that $\cH^i(F)\in\cA'$. Thus by induction, we are done.
        \end{proof}
        
        \begin{corollary}\label{generatorA}
        	Let $\cA$ be the heart of a bounded t-structure on $\cC$. Suppose $G\in \cC$ is a classical generator and let $G':=\bigoplus_{i\in \ZZ}\cH^{i}_{\cA}(G)$. Then $G'$ is a generator of $\cA$ (Definition \ref{def-generator-abelian}).
        \end{corollary}
        
        \begin{proof}
        	If $G'$ is contained in a proper Serre subcategory $ \cA' \subsetneq \cA$, then $G$ is contained in the thick closure of $\cA'$, which is a proper thick subcategory of $\cC$ by Lemma \ref{thick closure}. Since $G\in \cC$ is a generator, this is a contradiction. 
        \end{proof}

	\begin{theorem}\label{theorem-main}
		Let $\cC$ be a geometric phantom or quasi-phantom category (Definition \ref{def-phantom}). Then $\cC$ does not admit a Noetherian or Artinian bounded t-structure. 
	\end{theorem}
	
	\begin{proof}
		Let $X$ be a quasi-compact, quasi-separated scheme such that $\cC \subset \Dp(X)$ is an admissible embedding.
		Let $(\cC^{\geq 0}, \cC^{\leq 0})$ be a Noetherian (Artinian) bounded t-structure and $\cA:= \cC^{\geq 0} \cap \cC^{\leq 0}$ be its heart. By Theorem \ref{theorem-classicalgenerator}, $\Dp(X)$ admits a classical generator. Since $\cC\subset \Dp(X)$ is admissible, $\cC$ admits a classical generator, denoted by $G$. Let 
		$$
		G'=\bigoplus_{i\in \ZZ}\cH^{i}_{\cA}(G) \in \cA. $$
		By Corollary \ref{generatorA}, $G'$ is a generator of $\cA$ (Definition \ref{def-generator-abelian}). 
		By Lemma \ref{welldefineddim}, since $\cA$ is Noetherian (Artinian), $\KGdim(\cA)$ is defined. Denote $\alpha=\KGdim(\cA)$. By Lemma \ref{predecessor}, $\KGdim(\cA)=\alpha=(\alpha-1)+1$ has a predecessor $(\alpha-1)$. The Serre quotient 
		$$
		\cA_{\leq \alpha-1} \to \cA \twoheadrightarrow \cA/\cA_{\leq \alpha-1}
		$$
		induces an exact sequence on $K_{0}$:
		$$
		K_{0}(\cA_{\leq \alpha-1} )\to K_{0}(\cA) \to K_{0}(\cA/\cA_{\leq \alpha-1}) \to 0.
		$$
 By Definition \ref{def-KGdim},
 $\cA/\cA_{\leq \alpha-1}$ is a length category (i.e. generated by simple objects). Hence $K_{0}(\cA/\cA_{\leq \alpha-1})$ is nonzero and torsion free. Therefore $K_{0}(\cA)$ cannot be 0 or torsion.
	\end{proof}

	\section{References}
	
	\bibliographystyle{alpha}
	\renewcommand{\section}[2]{} 
	\bibliography{reference}

	\end{document}